\documentclass[12pt,oneside]{amsart}
\usepackage{amsmath,amsthm,amssymb,fullpage,amsfonts,tikz}
\usepackage{url}
\usepackage{hyperref}
\usepackage[hypcap]{caption}

\newcommand{\C}{\mathbb{C}}
\newcommand{\Q}{\mathbb{Q}}
\newcommand{\defn}[1]{\textbf{\emph{#1}}}

\newtheorem*{thm*}{Theorem} 
\newtheorem{conj}{Conjecture}
\newtheorem{lemma}{Lemma}
\theoremstyle{definition}
\newtheorem*{remark*}{Remark}

\def\arXiv#1{{\href{http://arxiv-web3.library.cornell.edu/abs/#1}{arXiv:#1}}}

\title{Proof of a conjecture of Kulakova et al. related to the $\mathfrak{sl}_2$ weight system}
\author{Dror Bar-Natan}
\address{
  Department of Mathematics\\
  University of Toronto\\
  Toronto Ontario M5S 2E4\\
  Canada
}
\email{drorbn@math.toronto.edu}
\urladdr{http://www.math.toronto.edu/~drorbn}

\author{Huan T. Vo}
\address{
  Department of Mathematics\\
  University of Toronto\\
  Toronto Ontario M5S 2E4\\
  Canada
}
\email{vohuan@math.toronto.edu}

\begin{document}
\maketitle

\begin{abstract}
In this article, we show that a conjecture raised in \cite{sl2Lando}, which regards the coefficient of the highest term when we evaluate the $\mathfrak{sl}_2$ weight system on the projection of a diagram to primitive elements, is a consequence of the Melvin-Morton-Rozansky conjecture, proved in \cite{MMR}.  
\end{abstract}

\section{Introduction}
In this section, we briefly recall a conjecture of \cite{sl2Lando} together with the relevant terminologies. A more complete treatment can be found in \cite{sl2Lando}. Given a chord diagram $D$ with $m$ chords, its \defn{labelled intersection graph} $\Gamma(D)$ is the simple labelled graph whose vertices are the chords of $D$, labelled from $1$ to $m$, and two vertices are connected by an edge if the two corresponding chords intersect. 

Following \cite{sl2Lando}, by orienting the chords of $D$ arbitrarily, we can turn $\Gamma(D)$ into an oriented graph as follows. Given two intersecting oriented chords $a$ and $b$, the edge connecting $a$ and $b$ goes from $a$ to $b$ if the beginning of the chord $b$ belongs to the arc of the outer circle of $D$ which starts at the tail of $a$ and goes in the positive (counter-clockwise) direction to the head of $a$ (see Figure \ref{fig:orientation}). We also have another description of the orientation. Given two intersecting oriented chords $a$ and $b$, we look at the smaller arc of the outer circle of $D$ that contains the tails of $a$ and $b$. Then we orient the edge connecting $a$ and $b$ from $a$ to $b$ if we go from the tail of $a$ to the tail of $b$ along the smaller arc in the counter-clockwise direction. The reader should check that the two definitions of orientation are equivalent. 

\begin{figure}[htb]
 \centering
\begin{tikzpicture}
      \draw [thick] (0,0) circle [radius=1];
      \draw [thick] [->] (-1,0)--(1,0);
      \draw [thick] [->] (0,-1)--(0,1);
      \node [left] at (-1,0) {$a$};
      \node [below] at (0,-1) {$b$};
      \draw [thick] [->] (2,0)--(3,0);
      \node [below] at (4,0) {$a$};
      \node [below] at (6,0) {$b$};
      \draw [thick] [->] (4,0)--(6,0);
 \end{tikzpicture}
 \caption{Orienting $\Gamma(D)$}
 \label{fig:orientation}
\end{figure}
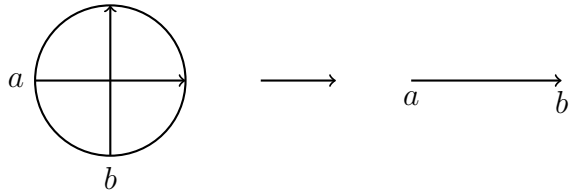

Consider a circuit of even length $l=2k$ in the oriented graph $\Gamma(D)$. By a \defn{circuit} we mean a closed path in $\Gamma(D)$ with no repeated vertices. Choose an arbitrary orientation of the circuit. For each edge, we assign a weight $+1$ if the orientation of the edge coincides with the orientation of the circuit and $-1$ otherwise. The \defn{sign} of a circuit is the product of the weights over all the edges in the circuit. We say that a circuit is \defn{positively oriented} if its sign is positive and \defn{negatively oriented} if its sign is negative. We define
  \[R_k(D)\colon=\sum_{s}\mathrm{sign} (s),\]      
where the sum	 is over all (un-oriented) circuits $s$ in $\Gamma(D)$ of length $2k$. 

\medskip

It is well-known that given a Lie algebra $\mathfrak{g}$ equipped with an ad-invariant non-degenerate bilinear form, we can construct a weight system $w_{\mathfrak{g}}$ with values in the center $ZU(\mathfrak(g))$ of the universal enveloping algebra $U(\mathfrak{g})$ (see, for instance \cite[\textbf{Section~6}]{VassilievDuzhin}). In the case of the Lie algebra $\mathfrak{sl}_2$, we obtain a weight system with values in the ring $\C[c]$ of polynomials in a single variable $c$, where $c$ is the Casimir element of the Lie algebra $\mathfrak{sl}_2$. Note that the Casimir element $c$ also depends on the choice of the bilinear form. For the case of $\mathfrak{sl}_2$, an ad-invariant non-degenerate bilinear form is given by 
   \[\langle x,y\rangle=\mathrm{Tr}(\rho(x)\rho(y)),\quad x,y\in \mathfrak{sl}_2,\]
where $\rho\colon \mathfrak{sl}_2\to \mathfrak{gl}_2$ is the standard representation of $\mathfrak{sl}_2$. Since $\mathfrak{sl}_2$ is simple, any invariant form is of the form $\lambda\langle\cdot,\cdot\rangle$ for some constant $\lambda$. If we let $c_{\lambda}$ be the corresponding Casimir element and $c=c_1$, then $c_{\lambda}=c/\lambda$. Therefore if $D$ is a chord diagram with $n$ chords, the weight system
  \[w_{\mathfrak{sl}_2}(D)=c^n+a_{n-1}c^{n-1}+\cdots+a_1c\]
and the weight system corresponding to $\lambda\langle\cdot,\cdot\rangle$  
 \[w_{\mathfrak{sl}_2,\lambda}(D)=c_{\lambda}^n+a_{n-1,\lambda}c_{\lambda}^{n-1}+\cdots+a_{1,\lambda}c_{\lambda}\]
are related by
  \[w_{\mathfrak{sl}_2,\lambda}(D)=\frac{1}{\lambda^n}\left.w_{\mathfrak{sl}_2}(D)\right|_{c=\lambda c_{\lambda}},\] 
or 
 \[a_{n-1}=\lambda a_{n-1,\lambda},\quad a_{n-2}=\lambda^2a_{n-2.\lambda},\ \dots \ a_2=\lambda^{n-2}a_{2,\lambda},\quad a_1=\lambda^{n-1} a_{1,\lambda}.\]
  
  \medskip
  
Let $\mathcal{A}$ be the vector space generated by all chord diagrams (modulo the 1-term and 4-term relations). Then $\mathcal{A}$ is graded by the degree (number of chords) of a chord diagram, i.e.
    \[\mathcal{A}=\bigoplus_{n=0}^{\infty} \mathcal{A}_n\]
where $\mathcal{A}_n$ is the space generated by all chord diagrams with $n$ chords (modulo the 1-term and 4-term relations). Recall that we can turn $\mathcal{A}$ into a bialgebra by defining the following comultiplication 
   \[ \Delta(D)=\sum_{V(D)=V_1\sqcup V_2} D\big|_{V_1}\otimes D\big|_{V_2},\]
 where the sum is taken over all ordered disjoint partitions of $V(D)$, the set of chords of $D$. (Note that $V_1$ or $V_2$ can be empty.) An element $D\in \mathcal{A}$ is called \defn{primitive} if 
   \[\Delta(D)=1\otimes D+D\otimes 1.\]        
The set of all primitive elements forms a vector subspace of $\mathcal{A}$, which we denote by $\mathcal{P}$. On the other hand, a chord diagram $D$ is called \defn{decomposable} if it can be written as a product $D=D_1\cdot D_2$ of two diagrams of smaller degrees. We let $\mathcal{D}$ denote the subspace spanned by the decomposable elements of $\mathcal{A}$. Note that $\mathcal{P}$ and $\mathcal{D}$ inherit a grading from $\mathcal{A}$ and $\mathcal{A}_n=\mathcal{P}_n\oplus \mathcal{D}_n$. 

Following \cite{PrimitiveLando} we now define a map which projects $\mathcal{A}$ onto $\mathcal{P}$. Let $D$ be a chord diagram with $n$ chords, $V=V(D)$ its set of chords. Then the map $\pi_n$ from the space of chord diagrams to its primitive elements is given by 
  \[\pi_n(D)=D-1!\sum_{\{V_1,V_2\}}\left.D\right|_{V_1}\cdot\left.D\right|_{V_2}+2!\sum_{\{V_1,V_2,V_3\}}\left.D\right|_{V_1}\cdot\left.D\right|_{V_2}\cdot\left.D\right|_{V_3}-\cdots,\]      
 where sums are taken over all unordered disjoint partitions of $V$ into non-empty subsets and $\left.D\right|_{V_i}$ denotes $D$ with only chords from $V_i$. If we change unordered partitions to ordered ones, we obtain 
    \begin{equation}\label{PrimProj}
    \pi_n(D)=D-\frac12\sum_{V=V_1\sqcup V_2}\left.D\right|_{V_1}\cdot\left.D\right|_{V_2}+\frac13\sum_{V=V_1\sqcup V_2\sqcup V_3}\left.D\right|_{V_1}\cdot\left.D\right|_{V_2}\cdot\left.D\right|_{V_3}-\cdots.
    \end{equation}
 It is shown (see \cite{PrimitiveLando}) that $\pi_n$ is indeed a projection from $\mathcal{A}_n$ onto $\mathcal{P}_n$ along $\mathcal{D}_n$, i.e. it takes each primitive element into itself and it takes all decomposable elements to zero. Now we are ready to state the conjecture raised in \cite{sl2Lando}.

\begin{conj}\label{sl2Conj}
 Let $D$ be a chord diagrams with $2m$ chords ($m\geq 1$), and $w_{\mathfrak{sl}_2,2}$ be the weight system associated with $\mathfrak{sl}_2$ and $2\langle\cdot,\cdot\rangle$. Then 
    \[w_{\mathfrak{sl}_2,2}(\pi_{2m}(D))=2R_m(D)c_2^m+\text{ terms of degree less than $m$ in $c_2$}.\]
\end{conj} 

\section{Proof of the conjecture}
The conjecture is a consequence of the Melvin-Morton-Rozansky (MMR) conjecture, which was proved in \cite{MMR}. We recall the statement of the MMR conjecture below. Let $J^k(q)$ be the ``framing independent'' colored Jones polynomial associated with the $k$-dimensional irreducible representation of $\mathfrak{sl}_2$. Set $q=e^h$, write $J^k(q)$ as power series in $h$:
  \[J^k=\sum_{n=0}^{\infty} J^k_n h^n.\]       
It is known that $J^k_n$ is given by (see \cite[\textbf{Theorem~6.14}]{QuantumInvariantOhtsuki} and \cite[\textbf{Section~11.2.3}]{VassilievDuzhin})
    \[J^k_n=\mathrm{Tr}\left(\left.w'_{\mathfrak{sl}_2}\right|_{c=\frac{k^2-1}{2}\cdot I_k}\right).\]
Here $I_k$ is the $k\times k$ identity matrix and $w'_{\mathfrak{sl}_2}$ is the ``deframing'' of the weight system $w_{\mathfrak{sl}_2}$ (see \cite[\textbf{Section~4.5.4}]{VassilievDuzhin}). For any chord diagram $D$ of degree $n$ (modulo the framing independent relation), the value $w'_{\mathfrak{sl}_2}(D)$ is a polynomial in $c$ of degree at most $\lfloor n/2\rfloor$ (see \cite[\textbf{Exercise~6.25}]{VassilievDuzhin}). It follows that $J^k_n$ is a polynomial in $k$ of degree at most $n+1$. Dividing $J^k_n$ by $k$ we then obtain 
  \[\frac{J^k}{k}=\sum_{n=0}^{\infty} \left(\sum_{0\leq j\leq n} b_{n,j}k^j\right) h^n,\]
where $b_{n,j}$ are Vassiliev invariants of order $\leq n$. We denote the highest order part of the colored Jones polynomial by   
    \[JJ\colon =\sum_{n=0}^{\infty} b_{n,n}h^n.\]
    
Next we recall the definition of the Alexander-Conway polynomial of link diagrams. The Conway polynomial $C(t)$ can be defined by the skein relation:
  \begin{enumerate}
   \item[(i)] $C(\mathrm{unknot})=1$,
   \item[(ii)] $C(L_+)-C(L_{-})=tC(L_0)$,
  where $L_+$, $L_{-}$ and $L_0$ are identical outside the regions consisting of a positive crossing, a negative crossing and no crossing, respectively.  
  \end{enumerate}
 The Alexander-Conway polynomial is a Vassiliev power series: 
  \[ \widetilde{C}(h)\colon=\frac{h}{e^{h/2}-e^{-h/2}}\left.C\right|_{t=e^{h/2}-e^{-h/2}}=\sum_{n=0}^{\infty}c_nh^n.\]
 Now we are ready to state the MMR conjecture, which was proved in \cite{MMR}. 
 
 \begin{thm*}
   With the notations as above, we have 
     \begin{equation}\label{MMRPoly}
     JJ(h)(K)\cdot \widetilde{C}(h)(K)=1
    \end{equation} 
   for any knot $K$.  
 \end{thm*}
 
 The proof of the MMR conjecture found in \cite{MMR}  consists of reducing the equality of Vassiliev power series to an equality of weight systems. Recall that a Vassiliev invariant $\nu$ of order $n$ gives us a weight system $W_n(\nu)$ of order $n$ by $W_n(\nu)(D)=\nu(K_D)$, where $D$ is a chord diagram of degree $n$ and $K_D$ is a singular knot whose chord diagram is $D$. Let 
   \[W_{JJ}\colon=\sum_{n=0}^{\infty} W_n(b_{n,n}) \ \text{and}\ W_C\colon=\sum_{n=0}^{\infty} W_n(c_n).\]
Then it is shown in \cite{MMR} that the equality $\eqref{MMRPoly}$ is equivalent to 
  \[W_{JJ}\cdot W_C=\mathbf{1}.\]
 Here $\mathbf{1}$ denotes the weight system that takes value 1 on the empty chord diagram and 0 otherwise. Recall also that the product of two weight systems is given by 
  \[W_1\cdot W_2(D)=m(W_1\otimes W_2)(\Delta(D)),\]
where $m$ denotes the usual multiplication in $\Q$. When $D$ is primitive, we have 
  \[0 =W_{JJ}\cdot W_C(D) =m(W_{JJ}\otimes W_C)(D\otimes 1+1\otimes D)=W_{JJ}(D)+W_C(D).\]    
Thus we obtain
  
  \begin{lemma}
    If $D$ is a chord diagram of degree $2m$, then 
          \[W_{JJ}(\pi_{2m}(D))=-W_C(\pi_{2m}(D)).\] 
  \end{lemma}  
  
  \medskip
  
 To prove conjecture \ref{sl2Conj}, we need the notion of \defn{logarithm} of a weight system (see \cite[\textbf{Chapter~6}]{GraphLando}). Let $w$ be a weight system and suppose $w$ can be written as $w=\mathbf{1}+w_0$, where $w_0$ vanishes on chord diagrams of degree 0. Then 
  \[\log w\colon =\log (\mathbf{1} +w_0)=w_0-\frac12 w_0^2+\frac13 w_0^3-\cdots\]
is well-defined since for each chord diagram we only have finitely many non-zero summands. 

\begin{lemma}
  Let $w$ be a multiplicative weight system, i.e. $w(D_1\cdot D_2)=w(D_1)w(D_2)$, and $w(\text{empty chord diagram})=1$. If $D$ is a chord diagram of degree $2m$, then
    \[(\log w)(D)=w(\pi_{2m}(D)).\]
\end{lemma}   

\begin{proof}
  From the definition of the logarithm of a weight system we have 
   \begin{align*}
     \log w &=\log (\mathbf{1}+(w-\mathbf{1}))\\
                &=(w-\mathbf{1})-\frac12(w-\mathbf{1})^2+\frac13(w-\mathbf{1})^3-\cdots
   \end{align*}
 Now if $D$ is a chord diagram, then $(w-\mathbf{1})(\text{empty chord diagram})=0$ and $(w-\mathbf{1})(D)=w(D)$ if $D$ has degree $>0$. Therefore, 
   \begin{align*}
   (w-1)^k(D)&=\sum_{V_1\sqcup V_2\sqcup\cdots\sqcup V_k=V(D)}w(\left. D\right|_{V_1})w(\left. D\right|_{V_2})\cdots w(\left. D\right|_{V_k})\\
                &=\sum_{V_1\sqcup V_2\sqcup\cdots\sqcup V_k=V(D)}w(\left. D\right|_{V_1}\cdot\left. D\right|_{V_2}\cdots \left. D\right|_{V_k}),
    \end{align*}
  where the sum is over ordered disjoint partition of $V(D)$ into non-empty subsets and the last equality follows from the multiplicativity of $w$. Comparing with equation $\eqref{PrimProj}$ we obtain our desired equality.  
\end{proof}

It is known that the weight system $W_C$ is multiplicative. Therefore for a chord diagram $D$ of degree $2m$, 
  \[(\log W_C)(D)=W_C(\pi_{2m}(D)).\]
  
  \medskip

Given an oriented circuit $H$ in a labelled intersection graph, we define the \defn{descent} $d(H)$ of the circuit to be the number of label-decreases of the vertices when we go along the circuit in the specified orientation. We have the following lemma.  

\begin{lemma}
  Given a chord diagram $D$ of degree $2m$, we have 
    \[2R_m(D)=\sum_{H}(-1)^{d(H)}=-(\log W_C)(D),\]
  where the sum is over all oriented circuits $H$ of length $2m$ in $\Gamma(D)$.   
\end{lemma}

\begin{proof}
  To prove the first equality, we show that by labelling the chords of $D$ appropriately, the labelled intersection graph $\Gamma(D)$ of $D$ has the property that the edges always go in the direction of increasing indices. To get a required labelling, we cut the outer circle of $D$ to obtain a long chord diagram and then we label the chords by integers $1,2,\dots,2m$ as we encounter them when we go from left to right in an increasing fashion. Then it's clear that a descent will correspond to an edge with weight $-1$. Every circuit $H$ will have two possible orientations $H_{+}$ and $H_{-}$. However, since the circuit has even length, $d(H_{+})$ and $d(H_{-})$ have the same parity and the first equality follows.
  
   The second equality is proved in \cite[\textbf{Proposition~3.13}]{MMR}. Here we briefly describe the main idea for completeness. The key identity is $W_C(D)=\det \mathrm{IM} (D)$, where $\mathrm{IM} (D)$ is the intersection matrix of the chord diagram $D$, which is defined as follows: label the chords of $D$ as above, then $IM(D)$ is the $2m\times 2m$ matrix given by 
    \[ \mathrm{IM}(D)_{ij}=
     \begin{cases}
      \mathrm{sign}(i-j) & \text{if chords $i$ and $j$ of $D$ intersect,}\\
      0 & \text{otherwise.}
     \end{cases}
    \]
  It turns out that $\mathrm{IM}(D)$ only depends on $\Gamma(D)$. The identity is proved by showing that $\det \mathrm{IM}$ satisfies the defining relations of $W_C$ (see \cite[\textbf{Theorem~3}]{MMR}). Now expanding $\det \mathrm{IM}(D)$ we obtain 
    \[W_C(D)=\sum_{H=\bigcup_{\alpha} H_{\alpha}} (-1)^{\mathrm{sign}(\sigma_H)}(-1)^{d(H)}.\]
 Here $H$ is an oriented circuit of length $2m$ in $\Gamma(D)$, $\sigma_H$ is the permutation of the vertices of $\Gamma(D)$ underlying $H$ and $\bigcup_{\alpha} H_{\alpha}$ is the (unordered) cycle decomposition of $\sigma_H$. From there it follows that 
   \[(\log W_C)(D)=-\sum_H (-1)^{d(H)}.\]
 The readers can consult \cite{MMR} for more details.      
\end{proof}

\begin{proof}[Proof of Conjecture \ref{sl2Conj}] 
Let $D$ be a chord diagram of degree $2m$, we have a chain of equalities from the above lemmas
 \[2R_m(D)=\sum_{H} (-1)^{d(H)}=-(\log W_C)(D)=-W_C(\pi_{2m}(D))=W_{JJ}(\pi_{2m}(D)).\]
Therefore,  
  \[\frac{J_{2m}^k(\pi_{2m}(D))}{k}=2R_m(D)k^{2m}+\text{ terms of degree less than $2m$ in $k$}.\]
Plug in $c=(k^2-1)/2$ or $k^2=2c+1$ we obtain   
  \[w'_{\mathfrak{sl}_2}(\pi_{2m}(D))=2^{m+1}R_m(D)c^m+\text{ terms of degree less than $m$ in $c$}.\]
Note that for any primitive diagram $P$ of degree $>1$, we have $w'(P)=w(P)$, where $w'$ is the deframing of a weight system $w$ (see, for instance, \cite[\textbf{Exercise~4.8}]{VassilievDuzhin}). Since $\pi_{2m}(D)$ is primitive (of degree $2m>1$), we obtain     
  \[w_{\mathfrak{sl}_2}(\pi_{2m}(D))=2^{m+1}R_m(D)c^m+\text{ terms of degree less than $m$ in $c$}.\]
 Now we just need to do a change of variable
  \begin{align*}
  w_{\mathfrak{sl}_2,2}(\pi_{2m}(D))&=\frac{1}{2^{2m}}\left.w_{\mathfrak{sl}_2}(\pi_{2m}(D))\right|_{c=2c_2}\\
                                                                &=2R_m(D)c_2^m+\text{ terms of degree less than $m$ in $c_2$},
  \end{align*}
which completes the proof. 
\end{proof}

\section*{Acknowledgements} We wish to thank E. Kulakova, S. Lando, T. 
Mukhutdinova, G. Rybnikov for communications regarding this paper 
and our referees for additional comments. The first author was 
partially supported by NSERC grant RGPIN 262178.


\end{document}